\documentclass[11pt]{article} 
\usepackage{amssymb,amsfonts,amsmath,amsthm,amscd,stackrel,dsfont,mathrsfs}
\RequirePackage[colorlinks,citecolor=blue,urlcolor=blue]{hyperref}

\newcommand{\red}[1]{#1}

\newcommand{\eqnsection}{\renewcommand{\theequation}{\thesection.\arabic{equation}}
      \makeatletter \csname @addtoreset\endcsname{equation}{section}\makeatother}

\newtheorem{theorem}{Theorem}[section]

\newtheorem{lemma}[theorem]{Lemma}
\newtheorem{proposition}[theorem]{Proposition}

\newtheorem{remark}[theorem]{Remark}

\newcommand{\be}{\begin{equation}}
\newcommand{\ee}{\end{equation}}
\newcommand{\beq}{\begin{equation*}}
\newcommand{\eeq}{\end{equation*}}
\newcommand{\bq}{\begin{eqnarray}}
\newcommand{\eq}{\end{eqnarray}}

\def\P{{\mathbb{P}}}

\newcommand{\eps}{\varepsilon}

\newcommand{\EE}{\mathbb E}

\newcommand{\RR}{\mathbb R}
\def\reals{\RR}

\newcommand{\PP}{\mathbb P}

\newcommand{\cK}{{\mathcal K}}

\newcommand{\cN}{{\mathcal N}}

\newcommand{\abbr}[1]{{\sc\lowercase{#1}}}

\begin{document}

\title{Persistence of iterated partial sums}
\author{{\sc Amir Dembo}\thanks{The research was supported in part by NSF grants DMS-0806211 \red{and DMS-1106627}.}
\and \red{{\sc Jian Ding}\thanks{This research partially done when Jian Ding was at MSRI program in 2012.}}
\and {\sc Fuchang Gao}\thanks{This research partially done when Fuchang Gao visited Stanford university in March 2010.
The authors also thank NSF grant DMS-0806211 for the support.}}
\date{\today}
\maketitle
\begin{abstract}
Let $S_n^{(2)}$ denote the iterated partial sums. That is,
$S_n^{(2)}=S_1+S_2+\cdots +S_n$, where $S_i=X_1+X_2+\cdots+X_i$.
Assuming $X_1, X_2,\ldots,X_n$ are integrable, zero-mean,
i.i.d. random variables, we show that the persistence probabilities
$$p_n^{(2)}:=\PP\left(\max_{1\le i \le n}S_i^{(2)}<  0\right) \le c\sqrt{\frac{\EE|S_{n+1}|}{(n+1)\EE|X_1|}},$$
with $c \le 6 \sqrt{30}$ (and $c=2$ whenever $X_1$ is symmetric).
\red{The converse inequality holds whenever the non-zero
$\min(-X_1,0)$ is bounded or when it has only finite third
moment and in addition $X_1$ is squared integrable. Furthermore,
$p_n^{(2)}\asymp n^{-1/4}$ for any non-degenerate squared integrable,
i.i.d., zero-mean $X_i$.}
In contrast, we show that for any $0 < \gamma < 1/4$ there exist
integrable, zero-mean random variables for which the rate of
decay of $p_n^{(2)}$ is $n^{-\gamma}$.
\end{abstract}

\section{Introduction}
The estimation of probabilities of rare events is one of the
central themes of research in the theory of probability. Of
particular note are \emph{persistence} probabilities, formulated as
\bq
q_n = \PP\Big(\max_{1\le k\le n} Y_k < y \Big),\label{max}
\eq
where $\{Y_k\}_{k=1}^n$ is a sequence of zero-mean random variables.
For independent $Y_i$ the persistence probability is easily
determined to be the product of $\PP(Y_k < y)$ and to a large
extent this extends to the case of sufficiently weakly dependent
and similarly distributed $Y_i$, where typically $q_n$ decays
exponentially in $n$. In contrast, in the classical case of partial sums,
namely $Y_k=S_k=\sum_{i=1}^k X_i$ with $\{X_j\}$ i.i.d.
zero-mean random variables, it is well known that $q_n = O(n^{-1/2})$
decays as a power law. This seems to be one of the very few cases
in which a power law decay for $q_n$ can be proved and its exponent
is explicitly known. Indeed, within the large class of similar
problems where dependence between $Y_i$ is strong enough to rule out
exponential decay, the behavior of $q_n$ is very sensitive to
the precise structure of dependence between the variables
$Y_i$ and even merely determining its asymptotic rate can be very
challenging (for example, see \cite{DPSZ} for recent results in case
$Y_k = \sum_{i=1}^n X_i (1-c_{k,n})^i$ are the values of a random
Kac polynomials evaluated at certain non-random $\{c_{k,n}\}$).

We focus here on iterated sums of i.i.d. zero-mean,
random variables $\{X_i\}$. That is, with $S_n=\sum_{k=1}^n X_k$ and
\bq
S^{(2)}_n=\sum_{k=1}^n S_k=\sum_{i=1}^n(n-i+1)X_i\,,
\label{S2}
\eq
we are interested in the asymptotics as $n \to \infty$
of the persistence probabilities
\bq
p_n^{(2)}(y):=\P\left(\max_{1\le k\le n}S_k^{(2)}<
y\right),\  \overline{p}_n^{(2)}(y):=\P\left(\max_{1\le k\le n}S_k^{(2)}\le
y\right)\,, \label{prob}
\eq
where $y \ge 0$ is independent of $n$.
With $y \ll n$ it immediately
follows from Lindeberg's \abbr{clt}
(when $X_i$ are square integrable),
that $p_n^{(2)}(y)\to 0$
as $n\to\infty$ and our goal is thus to find a sharp rate for
this decay to zero.

Note that for any fixed $y>0$ we have that
$\overline{p}_n^{(2)} (y) \asymp p_n^{(2)}(y) \asymp p_n^{(2)}(0)$
up to a constant depending only on $y$, here and throughout the paper, $A\asymp B$ means that there exists two positive constants $C_1$ and $C_2$, such that $C_1A\le B\le C_2A$.  Indeed, because
$\EE X_1^->0$, clearly $\PP(X_1<-\eps)>0$ for $\eps=y/k$
and some integer $k \ge 1$. Now, for any $n\ge 1$ and $z \ge 0$,
$$
\overline{p}^{(2)}_n(z) \ge
p_n^{(2)}(z) \ge \PP(X_1<-\eps)\overline{p}_{n-1}^{(2)}(z+\eps)\ge
\PP(X_1<-\eps) \overline{p}_{n}^{(2)}(z+\eps)
$$
and applying this inequality for $z=i \eps$, $i=0,1,\ldots,k-1$ we
conclude that
\bq
p_n^{(2)}(0)\ge [\PP(X_1<-\eps)]^k \overline{p}_n^{(2)}(y).\label{p0-py}
\eq
Of course, we also have the complementary trivial relations
$p_n^{(2)}(0)\le \overline{p}_n^{(2)} (0) \le p_n^{(2)}(y) \le
\overline{p}_n^{(2)}(y)$, so it suffices to consider
only $p_n^{(2)}(0)$ and $\overline{p}_n^{(2)}(0)$ which we
denote hereafter by $p_n^{(2)}$ and $\overline{p}_n^{(2)}$,
respectively. Obviously, $p_n^{(2)}$ and $\overline{p}_n^{(2)}$
have the same order (with $p_n^{(2)} = \overline{p}_n^{(2)}$
whenever $X_1$ has a density), and we consider both only
in order to draw the reader's attention to potential identities
connecting the two sequences $\{p_n^{(2)}\}$ and $\{\overline{p}_n^{(2)}\}$.

Persistence probabilities such as $p_n^{(2)}$ appear in many applications.
For example, the precise problem we consider here arises in the study
of the so-called sticky particle systems (c.f. \cite{Vysotsky2} and
the references therein). In case of standard normal $X_i$ it is also
related to entropic repulsion for $\nabla^2$-Gaussian fields (c.f.
\cite{grad2} and the references therein), though
here we consider the easiest version, namely a one dimensional
$\nabla^2$-Gaussian field. In his 1992 seminal paper, Sinai \cite{Sinai}
proved that if $\PP(X_1=1)=\PP(X_1=-1)=1/2$, then
$p_n^{(2)}\asymp n^{-1/4}$. However, his method relies on the fact
that for Bernoulli $\{X_k\}$ all local minima of $k \mapsto S_k^{(2)}$
correspond to values of $k$ where $S_k=0$, and as such form a
sequence of regeneration times. For this reason, Sinai's method
can not be readily extended to most other distributions. Using a
different approach, more recently Vysotsky \cite{Vysotsky} managed to
extend Sinai's result that $p_n^{(2)}\asymp n^{-1/4}$
to $X_i$ which are double-sided exponential, and a few other
special types of random walks. At about the same time,
Aurzada and Dereich \cite{AD} used strong approximation techniques
to prove the bounds $n^{-1/4}(\log n)^{-4}\lesssim p_n^{(2)}\lesssim
n^{-1/4}(\log n)^4$ for zero-mean random variables $\{X_i\}$
such that $\EE [e^{\beta |X_1|}] <\infty$ for some $\beta>0$.
However, even for $X_i$ which are standard normal variables it was not
known before the present results whether these logarithmic terms are needed, and if not,
how to get rid of them. Our main result, stated below, \red{fully
resolves this question, requiring only that $\EE X_1^2$ is finite and positive.}

\begin{theorem}\label{thm-main}
For i.i.d. $\{X_k\}$ of zero mean and $\red{0<} \EE|X_1|<\infty$,
let $S_n^{(2)}=S_1+S_2+\cdots +S_n$, where $S_i=X_1+X_2+\cdots+
X_i$. Then,
\bq
\sum_{k=0}^np_{k}^{(2)}\overline{p}_{n-k}^{(2)}\le c_1^2\frac{\EE|S_{n+1}|}{\EE|X_1|},\label{upper-bound}
\eq
where $c_1\le 6 \sqrt{30}$, and $c_1=2$ if $X_1$ is symmetric. \red{The converse inequality}
\bq
\sum_{k=0}^{n} p_{k}^{(2)}p_{n-k}^{(2)} \ge \frac{1}{c_2}
\frac{\EE|S_{n+1}|}{\EE|X_1|}
\label{lower-bound}
\eq
holds for some finite $c_2$ \red{whenever $X_1^-$ is bounded, or
with $X_1^-$ only having finite third moment and $X_1$ squared integrable}.
Taken together, these bounds imply that
\bq
\frac{1}{4c_1c_2}\sqrt{\frac{\EE|S_{n+1}|}{(n+1)\EE|X_1|}}\le p_n^{(2)}\le c_1\sqrt{\frac{\EE|S_{n+1}|}{(n+1)\EE|X_1|}}.
\label{two-side}
\eq
\red{Furthermore, assuming \emph{only} that $\EE X_1=0$ and $0< \EE (X_1^2) < \infty$,
we have that
\bq\label{quarter-exp}
p_n^{(2)}\asymp n^{-1/4} \,.
\eq
}
\end{theorem}
\begin{remark}\label{Bernstein}
In contrast \red{to (\ref{quarter-exp})}, for
any $0< \gamma < 1/4$ there exists integrable, zero-mean variable
$X_1$ for which $p_n^{(2)}\asymp n^{-\gamma}$. Indeed, considering
$\PP(Y_1 > y) = y^{-\alpha} 1_{y \ge 1}$ with $1 < \alpha <2$,
\red{the bounds (\ref{two-side}) hold}
for the bounded below, zero-mean,
integrable random variable $X_1=Y_1-\EE Y_1
$. Setting $a_n=n^{1/\alpha}$,
clearly $n \PP(|X_1|>a_n x) \to x^{-\alpha}$ as $n \to \infty$,
hence $a_n^{-1} S_n - b_n$ converges in distribution to
a zero-mean, one-sided Stable$_\alpha$ variable $Z_\alpha$,
and it is further easy to check that
$b_n = a_n^{-1} n \EE[X_1 1_{|X_1| \le a_n}] \to b_\infty = -\EE Y_1$.
In fact, it is not hard to verify that
$\{a_n^{-1} S_n\}$ is a uniformly integrable sequence and
consequently $n^{-1/\alpha} \EE |S_n| \to \EE |Z_\alpha - \EE Y_1|$
finite and positive. From Theorem \ref{thm-main} we then
deduce that $p_n^{(2)}\asymp n^{-\gamma}$ for $\gamma=(1-1/\alpha)/2$.
This rate matches with the corresponding one for integrated
L\'{e}vy $\alpha$-stable process, c.f. \cite{Simon2}.
\end{remark}

The sequences $\{S_k\}$ and $\{S_k^{(2)}\}$ are special cases
of the class of auto-regressive processes
$Y_k = \sum_{\ell=1}^L a_{\ell} Y_{k-\ell} + X_k$
with zero initial conditions, i.e. $Y_k \equiv 0$ when $k \le 0$
(where $S_k$ corresponds to $L=a_1=1$ and $S_k^{(2)}$ corresponds
to $L=a_1=2$, $a_2=-1$).
While for such stochastic processes $(Y_k,\ldots,Y_{k-L+1})$
is a time-homogeneous Markov chain of state space $\reals^L$ and
$q_n = \PP(\tau>n)$ is merely the upper tail of the first
hitting time $\tau$ of $[y,\infty)$ by the first coordinate
of the chain, the general theory of Markov chains does not provide the
precise decay of $q_n$, which even in case $L=1$ ranges from
exponential decay for $a_1>0$ small enough (which can be proved by
comparing with O-U process,
c.f. \cite{AB11}),
via the $O(n^{-1/2})$ decay for $a=1$ to a constant
$n \mapsto q_n$ in the
limit $a_1 \uparrow \infty$. While we do not pursue this here, we
believe that the approach we develop for proving Theorem \ref{thm-main}
can potentially determine the asymptotic behavior of $q_n$ for a
large collection of auto-regressive processes. This is of much
interest, since for example, as shown in \cite{Li-Shao},
the asymptotic tail probability that random Kac polynomials
have no (or few) real roots is determined in terms of
the limit as $r \to \infty$ of the power law tail decay
exponents for the iterates $S_k^{(r)} = \sum_{i=1}^k S_i^{(r-1)}$, $r \ge 3$.

Our approach further suggests that there might
be some identities connecting the sequences
$\{p_n^{(2)}\}$ and $\{\overline{p}_n^{(2)}\}$. Note that, if we denote
$$p_n^{(1)}=\PP\left(\max_{1\le k\le n}S_k<0\right),
\quad \overline{p}_n^{(1)}=\PP\left(\max_{1\le k\le n}S_k\le 0\right),$$
then as we show in the proof of the following proposition that
there are indeed identities connecting the sequences
$\{p_n^{(1)}\}$ and $\{\overline{p}_n^{(1)}\}$.

\begin{proposition}\label{pn1}
If $X_i$ are mean zero i.i.d. symmetric random variables
then for all $n \geq 1$,
 \bq
p_n^{(1)} \le \frac{(2n-1)!!}{(2n)!!}\le \overline{p}_n^{(1)}\,.
\label{discrete}
\eq
In particular, if $X_1$ also has a density, then
\bq
p_n^{(1)} = \frac{(2n-1)!!}{(2n)!!}\,.\label{p1}\eq
\end{proposition}

\begin{remark} Proposition \ref{pn1} is not new and can be
found in \cite[Section XII.8]{Feller}.
In fact, it is \red{shown there that for all zero-mean random variables
with bounded second moment (not necessary symmetric),
\bq\label{eq:crude}
p_n^{(1)} \asymp n^{-1/2} \,.
\eq
}
The novel point is our elegant proof, which serves as the starting point
of our approach to the study of $p_n^{(2)}$.
\end{remark}

\begin{remark}
Let $B(s)$ denote a Brownian motion starting at $B(0)=0$ and
consider the integrated Brownian motion $Y(t)=\int_0^t B(s) ds$.
Sinai \cite{Sinai} proved the existence of positive constants
$A_1$ and $A_2$ such that for any $T \red{\ge 1}$,
\begin{equation}\label{eq:sinai-bm}
A_1 T^{-1/4}\le \PP\Big(\sup_{t\in [0,T]} Y(t) \le 1\Big)
\le A_2 T^{-1/4}.
\end{equation}
Upon setting $\eps = T^{-3/2}$ and $t=uT$,
by Brownian motion scaling this is equivalent up to a
constant to the following result that can be derived from an
implicit formula of McKean \cite{McKean} (c.f. \cite{Simon}):
$$
\lim_{\eps\to 0^+}\eps^{-1/6}\PP\Big(\sup_{u\in [0,1]} Y(u)
\le \eps\Big)=\frac{3\Gamma(5/4)}{4\pi\sqrt{2\sqrt{2\pi}}}.
$$
Since the iterated partial sums $S_n^{(2)}$
corresponding to i.i.d. standard normal random
variables $\{X_i\}$, forms a ``discretization" of $Y(t)$,
the right-most inequality in (\ref{eq:sinai-bm})
readily follows from Theorem \ref{thm-main}.
Indeed, with $\EE [Y(k)Y(m)]=k^2 (3m-k)/6$ and
$\EE[S_k^{(2)} S_m^{(2)}]=k(k+1)(3m-k+1)/6$ for
$m \ge k$, setting $Z(k)=\sqrt{(1+1/k)(1+1/(2k)}Y(k)$,
results with $\EE [(S_k^{(2)})^2]=\EE[Z(k)^2]$ and it
is further not hard to show that
$f(m,k):= \EE [S_m^{(2)}S_k^{(2)}]/\EE [Z(m)Z(k)] \ge 1$
for all $m \ne k$ (as $f(k+1,k) \ge 1$ and $d f(m,k)/dm >0$
for any $m \ge k+1$). Thus, by Slepian's lemma, we have that for any $y$
$$
\PP\Big(\max_{1\le k\le n} Z(k) <  y\Big)\le p_n^{(2)}(y) \,,
$$
and setting $n$ as the integer part of $T\ge 1$ it follows that
$$
\PP\Big(\sup_{t\in [0,T]} Y(t) \le 1 \Big)\le
\PP\Big(\max_{1\le k\le n}Y(k)\le 1\Big)
\le \PP\Big(\max_{1\le k\le n}Z(k) < 2 \Big) \le p_n^{(2)}(2) \,.
$$
Since $p_n^{(2)}(2) \le c p_n^{(2)}$ for some finite constant $c$ and all
$n$, we conclude from Theorem \ref{thm-main} that
$$
\PP\Big(\sup_{t\in [0,T]} Y(t) \le 1\Big) \le 2 c (n+1)^{-1/4}\le 2 c T^{-1/4}\,.
$$
\end{remark}

\section{Proof of Proposition \ref{pn1}}\label{pn1-proof}
Setting $S_0 =0$ let $M_n=\max_{0 \le j \le n} S_j$ and
consider the $\{0, 1,2,\ldots,n\}$-valued random variable
$$
\cN = \min\left\{ l \ge 0 : S_l=M_n \right\} \,.
$$
For each $k=1,2,\ldots,n-1$ we have that
\begin{align}
\{\cN=k\}&=\{X_k>0, X_{k}+X_{k-1}>0,\ldots,
X_k+X_{k-1}+\cdots+X_1>0;
\nonumber \\
& X_{k+1}\le 0, X_{k+1}+X_{k+2}\le
0, \ldots, X_{k+1}+X_{k+2}+\cdots+X_n\le 0\}.
\nonumber
\end{align}
By the independence of $\{X_i\}$, the latter identity implies that
\begin{align}
\P(\cN=k)&=\P(X_k>0, X_{k}+X_{k-1}>0,\ldots, X_k+X_{k-1}+\cdots+X_1>0)
\nonumber \\
\nonumber
&\times
\P(X_{k+1}\le 0, X_{k+1}+X_{k+2}\le 0, \ldots, X_{k+1}+X_{k+2}+\cdots+X_n\le 0)\\
&=p_k^{(1)}\overline{p}_{n-k}^{(1)},
\nonumber
\end{align}
where the last equality follows
from our assumptions that $X_i$ are i.i.d. symmetric random
variables.  Also note that $\PP(\cN=0)=\overline{p}_n^{(1)}$
and
$$\PP(\cN=n)=\P(X_n>0,
X_{n}+X_{n-1}>0,\ldots, X_n+X_{n-1}+\cdots+X_1>0)=p_n^{(1)}.$$
Thus, setting $p_0^{(1)}=\overline{p}_0^{(1)}=1$ we arrive at
the identity
\bq
\sum_{k=0}^{n}p_{k}^{(1)}\overline{p}_{n-k}^{(1)}=\sum_{k=0}^n\P(\cN=k)=1,\label{identity}
\eq
holding for all $n\ge 0$.

Fixing $x \in [0,1)$, upon multiplying (\ref{identity})
by $x^n$ and summing over $n \ge 0$, we arrive
at $P(x) \overline{P}(x) = \frac1{1-x}$, where
$P(x)=\sum_{k=0}^\infty p_k^{(1)}x^k$ and
$\overline{P}(x)=\sum_{k=0}^\infty \overline{p}_k^{(1)}x^k$.
Now, if $X_1$ also has a density then $p_k^{(1)}=\overline{p}_k^{(1)}$ for all $k$
and so by the preceding $P(x)=\overline{P}(x)=(1-x)^{-1/2}$. Consequently,
$p_n^{(1)}$ is merely the coefficient of $x^n$ in the Taylor
expansion at $x=0$ of the function $(1-x)^{-1/2}$, from which
we immediately deduce the identity (\ref{p1}).

If $X_1$ does not have a density, let $\{Y_i\}$ be i.i.d. standard normal random variables, independent of the sequence
$\{X_i\}$ and denote by $S_k$ and $\tilde{S}_k$ the partial sums of $\{X_i\}$ and $\{Y_i\}$, respectively.
Note that for any $\eps>0$, each of the i.i.d. variables $X_i+\eps Y_i$ is symmetric and has a density,
with the corresponding partial sums being $S_k+\eps \tilde{S}_k$. Hence, for any $\delta>0$ we have that
\begin{align}
\PP\Big(\max_{1\le k\le n} S_k<-\delta\Big)&\le
\PP\Big(\max_{1\le k\le n} (S_k+\eps \tilde{S}_k)\le 0\Big)+\PP\Big(\max_{1\le k\le n} \eps\tilde{S}_k \ge\delta\Big)
\nonumber\\
&=\frac{(2n-1)!!}{(2n)!!}+\PP\Big(\max_{1\le k\le n} \eps\tilde{S}_k \ge\delta\Big).
\nonumber
\end{align}
Taking first $\eps\downarrow 0$ followed by $\delta\downarrow 0$, we conclude that
$$
\PP\Big(\max_{1\le k\le n} S_k<0\Big)\le \frac{(2n-1)!!}{(2n)!!}\,,
$$
and a similar argument works for the remaining inequality in (\ref{discrete}).

\section{Proof of Theorem \ref{thm-main}}\label{main-proof}

By otherwise considering $X_i/\EE|X_i|$, we assume without loss of
generality that $\EE|X_1|=1$.
To adapt the method of Section \ref{pn1-proof} for dealing with
the iterated partial sums $S_n^{(2)}$, we introduce the parameter $t\in \RR$ and
consider the iterates $S_j^{(2)}(t) =S_0(t) +\cdots+S_j(t)$, $j \ge 0$,
of the translated partial sums $S_k(t) =t
+S_k$, $k \ge 0$.
That is, $S_j^{(2)}(t)=(j+1) t + S_j^{(2)}$ for each $j \ge 0$.

Having fixed the value of $t$, we define the following
$\{0,1,2,\ldots,n\}$-valued random variable
$$
\cK_t=\min\Big\{l \ge 0: S_l^{(2)}(t) =\max_{0\le j\le n} S_j^{(2)}(t)\Big\}.
$$
Then, for each $k=2,3,\ldots,n-2$, we have the identity
\begin{align}
\{\cK_t=k\}&=\left\{S_k(t) >0,
S_{k}(t) +S_{k-1}(t) >0,\ldots,
S_k(t) +S_{k-1}(t) +\cdots+S_1(t) >0;\right.
\nonumber
\\&\ \
\left. S_{k+1}(t) \le 0, S_{k+1}(t) +S_{k+2}(t) \le 0, \ldots,
S_{k+1}(t) +S_{k+2}(t) +\cdots+S_n(t) \le 0\right\}
\nonumber \\
&= \left\{S_{k}(t)>0; X_{k}<2S_k(t),\ldots,(k-1)X_k+\cdots+X_2<kS_k(t)\right\} \cap
\{S_{k+1}(t)\le 0\}
\nonumber
\\&\cap\left\{X_{k+2}\le -2S_{k+1}(t), \ldots,
(n-k-1)X_{k+2}+\cdots +X_n\le -(n-k)S_{k+1}(t)\right\}.
\nonumber
\end{align}
Next, for $2 \le k \le n$ we define $Y_{k,2} \in \sigma(X_2,\ldots,X_k)$ and
$Y_{k,n} \in \sigma(X_k,\ldots,X_n)$ such that
\begin{align}
\nonumber
&Y_{k,2}=\max\left\{\frac{X_k}{2},\frac{2X_k+X_{k-1}}{3}, \ldots, \frac{(k-1)X_{k}+\cdots+X_2}{k}\right\},\\
&
Y_{k,n}=\max\left\{\frac{X_k}{2},\frac{2X_k+X_{k+1}}{3}, \ldots, \frac{(n-k+1)X_k+\cdots+X_n}{n-k+2}\right\}.
\nonumber
\end{align}
It is then not hard to verify that the preceding identities translate into
\begin{align}
\{\cK_t=k\}&= \{S_k(t)>0 \ge S_{k+1}(t) \} \cap \{Y_{k,2}<S_k(t)\}\cap\{Y_{k+2,n}\le -S_{k+1}(t)\} \label{ID} \\
&= \{ -S_k + (Y_{k,2})^+ < t \le -X_{k+1} - S_k - (Y_{k+2,n})^+ \}
\label{ID2}
\end{align}
holding for each $k=2,\ldots,n-2$. Further, for $k=1$ and $k=n-1$ we have that
\begin{align}
\{\cK_t=1\}  &=\{S_1(t)>0\}\cap \{S_{2}(t)\le 0\}\cap\{Y_{3,n}\le -S_{2}(t)\}\,,\nonumber \\
\{\cK_t=n-1\}&=\{S_{n-1}(t)>0\}\cap\{Y_{n-1,2}<S_{n-1}(t)\}\cap \{S_{n}(t)\le 0\}\,,\nonumber
\end{align}
so upon setting $Y_{1,2}=Y_{n+1,n}=-\infty$, the identities (\ref{ID}) and (\ref{ID2})
extend to all $1 \le k \le n-1$.

For the remaining cases, that is, for $k=0$ and $k=n$, we have instead that
\begin{align}
\{\cK_t=0\}&=\{t\le -X_1 - (Y_{2,n})^+ \}\,,\label{0}\\
\{\cK_t=n\}&=\{-S_n+(Y_{n,2})^+ < t\}\,.\label{n}
\end{align}

In contrast with the proof of Proposition \ref{pn1}, here we have
events $\{(Y_{k,2})^+<S_k(t)\}$ and $\{(Y_{k+2,n})^+\le -S_{k+1}(t)\}$
that are linked through $S_{k}(t)$ and consequently
not independent of each other. Our goal is to unhook
this relation and in fact the parameter $t$
was introduced precisely for this purpose.

\subsection{Upper bound}
For any integer $n>1$, let
$$A_n =\max_{1\le k\le n}\{-S_{k+1} \},\ \ B_n =-\max_{1\le k\le n}\{S_k \}.$$
By definition $A_n\ge B_n$. Further, for any $1 \le k \le n-1$, from (\ref{ID}) we have that the event $\{\cK_t=k\}$ implies that
$\{S_k(t) > 0 \ge S_{k+1}(t) \} = \{ -S_k < t \leq - S_{k+1} \}$
and hence that $\{B_{n-1} <t \le A_{n-1} \}$. From (\ref{ID2}) we also
see that for any $1 \le k \le n-1$,
$$
\int_{\RR} 1_{\{\cK_t=k\}}dt \ge
(X_{k+1})^- 1_{\{Y_{k,2}<0\}}1_{\{Y_{k+2,n}\le 0\}}
$$
and consequently,
\begin{align}
A_{n-1}-B_{n-1} = \int_{\RR} 1_{\{B_{n-1} <t \le A_{n-1} \}}dt &\ge
\sum_{k=1}^{n-1} \int_{\RR} 1_{\{\cK_t=k\}}dt \label{sum-indicator}\\
&\ge \sum_{k=1}^{n-1} (X_{k+1})^- \, 1_{\{Y_{k,2}<0\}}1_{\{Y_{k+2,n}\le 0\}}\,. \nonumber
\end{align}
Taking the expectation of both sides
we deduce from the mutual independence of $Y_{k,2}$, $X_{k+1}$ and
$Y_{k+2,n}$ that
$$
\EE [A_{n-1}-B_{n-1}] \ge \sum_{k=1}^{n-1}\EE [(X_{k+1})^-] \PP(Y_{k,2}<0) \PP(Y_{k+2,n}\le 0) \,.
$$
Next, observe that since the sequence $\{X_i\}$ has an exchangeable law,
\begin{align}
\PP(Y_{k,2} < 0) &= \PP(X_k < 0, 2 X_k + X_{k-1} < 0, \ldots, (k-1) X_k + \cdots + X_2 < 0)
\nonumber \\
&= \PP(X_1<0, 2 X_1 + X_2 < 0, \ldots, (k-1) X_1 + \cdots + X_{k-1} < 0) = p_{k-1}^{(2)} \,.
\label{basic-bd}
\end{align}
Similarly, $\PP(Y_{k+2,n} \le 0) = \overline{p}_{n-1-k}^{(2)}$.
With $X_{k+1}$ having
zero mean, we have that $\EE [(X_{k+1})^-] = \EE[ (X_{k+1})^+] = 1/2$ (by
our assumption that $\EE |X_{k+1}| = \EE |X_1| =1$). Consequently,
for any $n>2$,
$$
\EE [A_{n-1}-B_{n-1}] \ge \frac{1}{2} \sum_{k=1}^{n-1} p_{k-1}^{(2)}\overline{p}_{n-1-k}^{(2)}
= \frac{1}{2} \sum_{k=0}^{n-2} p_k^{(2)} \overline{p}_{n-2-k}^{(2)} \,.
$$
With $\EE[S_{n+1}]=0$ and $\{X_k\}$ exchangeable, we clearly have that
\begin{equation}\label{an-bn:bd}
\EE[A_n-B_n] = \EE[\max_{1 \le k \le n} \{S_{n+1}-S_{k+1}\}]+\EE[\max_{1 \le k \le n} S_k]
=2 \EE[\max_{1 \le k \le n} S_k] \,.
\end{equation}
Recall Ottaviani's maximal inequality that for a symmetric random walk
$\PP(\max_{k=1}^n S_k \ge t) \le 2 \PP(S_n \ge t)$ for any $n, t \ge 0$,
hence in this case
$$
\EE[\max_{1 \le k \le n} S_k] \le
2 \int_0^\infty \PP(S_n\ge t) dt = \EE |S_n| \,.
$$
To deal with the general case, we replace Ottaviani's maximal inequality
by Montgomery-Smith's inequality
$$
\PP(\max_{1\le k\le n}|S_k|\ge t)
\le 3\max_{1\le k\le n}\PP(|S_k|\ge t/3)\le 9\PP(|S_n| \ge t/30)
$$
(see \cite{Montgomery}), from which we deduce that
\begin{equation}\label{an-bn:bd2}
\EE[\max_{1 \le k \le n} S_k] \le
9 \int_0^\infty \PP(|S_n|\ge t/30) dt = 270 \EE |S_n|
\end{equation}
and thereby get (\ref{upper-bound}).
Finally, since $n \mapsto p_n^{(2)}$ is non-increasing and
$p_n^{(2)}\le\overline{p}_{n}^{(2)}$, the upper bound of
(\ref{two-side}) is an immediate consequence of (\ref{upper-bound}).

\subsection{Lower bound}
Turning to obtain the lower bound, let
$$m_n:= -X_1-(Y_{2,n})^+\,,\ \ M_n:=-S_n + (Y_{n,2})^+ \,.$$
Note that for any $n \ge 2$, by using the last term of the maxima in the definition of $Y_{n,2}$ and $Y_{2,n}$, we have
$$
Y_{n,2}+Y_{2,n} \ge \frac{1}{n} [(n-1) X_n + \cdots + X_2] + \frac{1}{n} [
(n-1) X_2 + \cdots + X_n] = S_n - X_1 \,,
$$
and consequently,
\begin{equation}\label{lbd-diff}
M_n-m_n \ge X_1-S_n + (Y_{2,n}+Y_{n,2})^+ \ge (X_1-S_n)^+ = (X_2+\cdots+X_n)^-\,.
\end{equation}
In particular, $M_n\ge m_n$. From (\ref{0}) and (\ref{n}) we know that if $m_n<t \le M_n$
then necessarily $1 \le \cK_t \le n-1$. Therefore,
\begin{equation}\label{L}
M_n-m_n = \int_{\RR} 1_{\{m_n < t \le M_n\}}dt \le
\sum_{k=1}^{n-1} \int_{\RR} 1_{\{\cK_t=k\}} dt \,.
\end{equation}
In view of (\ref{ID2}) we have that for any $1 \le k \le n-1$,
$$
b_k := \EE [\int_{\RR}1_{\{\cK_t=k\}}dt]
=\EE \Big[ \left(X_{k+1} + (Y_{k,2})^+ \, +(Y_{k+2,n})^+ \right)^-\Big].
$$
By the mutual independence of the three variables on the right side,
\red{
and since $\{X_k\}$ have identical distribution, we find that
\begin{align}
b_k &= \int_0^\infty \PP(X_{k+1} < - x, (Y_{k,2})^+\, + (Y_{k+2,n})^+ < x) dx
\nonumber \\
& \le \int_0^\infty \PP(-X_1 > x) \PP(Y_{k,2} < x) \PP(Y_{k+2,n} < x) dx \,.
\label{eq:intc2}
\end{align}
Next, setting $T_{i,k}^{(2)} = T_{1,k} + \cdots + T_{i,k}$ for
$T_{i,k}=X_k+\cdots+X_{k+1-i}$, $i \ge 1$ and $T_{0,k} := 0$,
observe that for any $0 \le j \le k-1$ and $\ell \ge 1$,
$$
T_{j+\ell,k+\ell}^{(2)} = T_{\ell-1,k+\ell}^{(2)}
+ (j+1) T_{\ell,k+\ell} + T_{j,k}^{(2)}\,.
$$
Hence, with $A_{\ell,k} := \{ T_{i,k}^{(2)} < 0,\, i=1,\ldots,\ell-1\}$,
just as we did in deriving the identity (\ref{basic-bd}),
we have that for any $\ell \ge 1$,
\begin{align*}
\{ Y_{k+\ell,2} < 0 \} &= \{  A_{\ell,k+\ell} , \quad
T_{\ell-1,k+\ell}^{(2)} + (j+1) T_{\ell,k+\ell} + T_{j,k}^{(2)} < 0,\,
0 \le j \le k-1 \} \,,
\\
\{ Y_{k,2} <  x \} &= \{T_{j,k}^{(2)} < (j+1) x, \, 1 \le j \le k-1\} \,.
\end{align*}
Consequently,
$$
\{ Y_{k,2} < x \} \bigcap \{ T_{\ell,k+\ell} < -x \} \bigcap A_{\ell,k+\ell}  \subseteq \{ Y_{k+\ell,2} < 0 \} \,.
$$
By exchangeability of $\{X_m\}$ we have that for any $k,\ell$,
$$
\PP(A_{\ell,k+\ell}) = p_{\ell-1}^{(2)} = \PP(Y_{\ell,2} < 0) \,.
$$
Thus, applying Harris's inequality for the non-increasing events
$A_{\ell,k+\ell}$ and $\{ T_{\ell,k+\ell} < -x \}$, we get by
the independence of $\{X_m\}$ that
$$
p_{k+\ell-1}^{(2)} = \PP(Y_{k+\ell,2} < 0) \ge \PP(Y_{k,2} < x)
\PP(T_{\ell,k+\ell} < -x) p_{\ell-1}^{(2)} \,.
$$
Since $T_{\ell,k+\ell}$ has the same law as $S_\ell$ we thus
get the bound
$$
\PP(Y_{k,2} < x) \le
\frac{p_{k+\ell-1}^{(2)}}{p_{\ell-1}^{(2)} \PP(S_\ell < -x)} \;,
$$
for any $\ell \ge 1$. Similarly, we have that for any $\ell \ge 1$,
$$
\PP(Y_{k+2,n} < x)
\le
\frac{p_{n-k+\ell-1}^{(2)}}{p_{\ell-1}^{(2)} \PP(S_\ell \le -x)} \;.
$$
Clearly $k \mapsto p_k^{(2)}$ is non-increasing, so
combining these bounds we find that
\bq\label{eq:bkbd}
b_k \le \frac{c_2}{2} p_{k}^{(2)} p_{n-k}^{(2)} \,,
\eq
for $c_2 := 2 \int_0^\infty \PP(-X_1 > x) g(x)^{-2} dx$, where
\bq\label{eq:gdef}
g(x):= \sup_{\ell \ge 1}  \{ p_{\ell-1}^{(2)} \, \PP(S_\ell < -x) \} \,.
\eq
For $(X_1)^-$ bounded it clearly suffices to show that
$g(x)>0$ for each fixed $x>0$, and this trivially holds
by the positivity of $\PP(X_1 < -r)$ for $r>0$ small enough
(hence, $p_{\ell}^{(2)} \ge \PP(X_1 < -r)^{\ell}$ also positive).
Assuming instead that $X_1$ has finite (and positive)
second moment, from (\ref{eq:crude}) and the trivial bound
$p_{\ell-1}^{(2)} \ge p_{\ell}^{(1)}$ we have that
for some $\kappa>0$ and all $x$,
$$
g(x) \ge \kappa \sup_{\ell \ge 1} \{ \, \frac{1}{\sqrt{\ell}} \, \PP(S_\ell < -x) \} \,.
$$
Further, by the \abbr{CLT}
there exists $M<\infty$ large enough such that
$\eta := \inf_{x > 0} \PP(S_{\lceil x M \rceil^2} < -x)$ is positive.
Hence, setting $\ell=\lceil x M \rceil^2$, we deduce that in
this case $g(x) \ge c/(1+x M)$ for some $c>0$ and all $x \ge 0$.
Consequently, $c_2$ is then finite provided
$$
3 M \int_0^\infty (1+x M)^2 \PP(-X_1 > x) dx \le \EE[(1+ (X_1)^- M)^3] < \infty\,,
$$
i.e. whenever $(X_1)^-$ has finite third moment.
Next, considering the expectation of both sides of (\ref{L})
we deduce that under the above stated conditions,}
for any $n>2$,
\begin{eqnarray*}
\EE(M_n-m_n) \le \frac{c_2}{2}
\sum_{k=1}^{n-1}p_{k-1}^{(2)}p_{n-k-1}^{(2)}\,.
\end{eqnarray*}
In view of (\ref{lbd-diff}) we also have that
$\EE(M_n-m_n)\ge \EE [(S_{n-1})^-] = \frac12 \EE|S_{n-1}|$, from which
we conclude that (\ref{lower-bound}) holds for all $n \ge 1$.

Turning to lower bound $p_n^{(2)}$ as in (\ref{two-side}),
recall that $n \mapsto p_n^{(2)}$ is non-increasing. Hence,
applying (\ref{lower-bound}) for $n=2m+1$ and utilizing the
previously derived upper bound of (\ref{two-side}) we have that
\begin{align}
\frac1{c_2} \EE|S_{2(m+1)}|&\le
2\sum_{k=0}^mp_k^{(2)}p_m^{(2)}\le 2c_1p_m^{(2)}
\sum_{k=0}^m\sqrt{\frac{\EE|S_{k+1}|}{k+1}}\nonumber
\\& \le 4 c_1 p_m^{(2)}\sqrt{(m+1)\EE|S_{m+1}|}\,,
\end{align}
where in the last inequality we use the fact that
for independent, zero-mean $\{X_k\}$, the sequence
$|S_k|$ is a sub-martingle, hence
$k \mapsto \EE|S_k|$ is non-decreasing.
This proves the lower bound of (\ref{two-side}).

\red{Our starting point for removing in (\ref{quarter-exp})
the finite third moment assumption on $(X_1)^-$ is
the following lemma which allows us to
consider in the sequel only $k=O(n)$.
\begin{lemma}\label{lem-jian}
For some $0< \epsilon, \delta < 1/2$, all $n\in \mathbb{N}$,
$m := \lceil \epsilon n \rceil$ and $|t| \le \epsilon \sqrt{n}$,
$$
\PP(m \le \cK_t \le n-m) \geq \delta\,.
$$
\end{lemma}
\begin{proof}
First, observe that for $|t| \le \epsilon \sqrt{n}$
by the definition of $\cK_t$ and $S^{(2)}_j(t)$,
\begin{align*}
\PP( \cK_t < m) & \le
\PP (\max_{0 \le j < m} S^{(2)}_j(t) \geq 2 m \sqrt{n})
+ \PP(\max_{0 \leq j \leq n} S^{(2)}_j(t) \leq 2 m \sqrt{n})\\
&
 \leq \PP (\max_{0 \le j \le m} S^{(2)}_j \geq \epsilon^{-1/2} m^{3/2})
+ \PP(\max_{0 \le j \leq n} S^{(2)}_j \leq 4 \epsilon n^{3/2})\,.
\end{align*}
For $b = \EE X_1^2$ finite and positive, by Donsker's invariance principle,
$n^{-3/2} \max_{0 \le j \leq n} S^{(2)}_j$ converge in law
as $n \to \infty$ to $\sqrt{b} \sup_{u \in [0,1]} Y(u)$. Hence,
by \eqref{eq:sinai-bm}, we deduce that
\begin{equation}\label{eq-1}
\lim_{\epsilon \downarrow 0} \lim_{n\to \infty} \PP( \cK_t < m)
= 0 \mbox{ uniformly for all } |t| \leq \epsilon \sqrt{n} \,.
\end{equation}
It remains to bound below $\PP(\cK_t \le n-m)$.
To this end, note that for $1 \le j \le m$,
$$
S^{(2)}_{j+n-m} (t) = S^{(2)}_{n-m} (t) + j t + j S_{n-m}
+ \widetilde{S}_j^{(2)} \,,
$$
where $\widetilde S^{(2)}_j = \sum_{i=1}^{j} \widetilde{S}_i$
and $\widetilde{S}_i =\sum_{\ell = 1}^{i} X_{\ell+n-m}$. Hence,
for $|t| \le \epsilon \sqrt{n}$,
\begin{align*}
\PP(\cK_t \le n - m) & \geq \PP(\cK_t \leq n -m,  S_{n-m} \leq - 2 \sqrt{n})
\\
&\geq \PP(S_{n-m} \leq - 2 \sqrt{n}) \,
\PP( \max_{1 \le j \le m} \{ \, \widetilde{S}^{(2)}_j - j \sqrt{n} \} < 0)\,.
\end{align*}
Clearly, if $\widetilde{S}_i < \sqrt{n}$ for all $i$ then necessarily
$\widetilde{S}^{(2)}_j < j \sqrt{n}$, from which we deduce that for any $m \le n/2$,
\begin{align}\label{eq:jian}
\PP(\cK_t \leq n - m) \geq \inf_{k \in [n/2,n]} \, \PP(S_k \leq - 2 \sqrt{n}) \,
\PP( \max_{1 \le j \leq n} \, \{S_j\} < \sqrt{n}) \,.
\end{align}
Since $n^{-1/2} \max\{ S_j : 1 \le j \leq n \}$ converges in law
to $\sqrt{b}$ times the absolute value of a
standard Gaussian variable, we conclude that as $n \to \infty$,
the right side of (\ref{eq:jian})
remains bounded away from zero, which in view of \eqref{eq-1}
yields our thesis.
\end{proof}

By Lemma \ref{lem-jian} we have that for $m=\lceil \epsilon n \rceil$ and
$n \in \mathbb{N}$,
$$
\delta \sqrt{\epsilon} \sqrt{m} \le
2 \delta \epsilon \sqrt{n} \le \sum_{k=m}^{n-m}
\EE \Big[ \int_{-\epsilon \sqrt{n}}^{\epsilon \sqrt{n}}
1_{\{\cK_t=k\}} dt \Big] \le \sum_{k=m}^{n-m} b_k \,.
$$
Further, the contribution to (\ref{eq:intc2})
from $x \in [L,\infty)$ is at most
$$
\int_L^\infty \PP(-X_1>x) dx = \EE[(X_1+L)^-] \le L^{-1} \EE[ X_1^2 1_{\{X_1^- \ge L\}}] \,.
$$
With $M$ as in the preceding bound on $c_2$,
set $L=L(m)=\sqrt{m}/(2M)$, noting that the total
contribution of these integrals to $\sum_{k=m}^{n-m} b_k$ is then at most
$$
\frac{2 M}{\epsilon} \sqrt{m} \EE[ X_1^2 1_{\{X_1^- \ge L(m) \}}] \,,
$$
which for some $m_0=m_0(\epsilon,\delta,M)$ finite and all
$m \ge m_0$ is further bounded above by
$(\delta/2) \sqrt{\epsilon} \sqrt{m}$. Consequently,
setting
$$
\kappa_m := \int_0^{L(m)} \PP(-X_1>x) g(x)^{-2} dx \,,
$$
we get by monotonicity of $k \mapsto p_k^{(2)}$ and the
arguments leading to (\ref{eq:bkbd}), that for $m \ge m_0$,
$$
\frac{\delta}{2} \sqrt{\epsilon} \sqrt{m} \le
\kappa_m \sum_{k=m}^{n-m} p_k^{(2)} p_{n-k}^{(2)}
\le \kappa_m \frac{m}{\epsilon} (p_m^{(2)})^2  \,.
$$
Setting now $p_\ell^{(2)} := \ell^{-1/4} \psi(\ell)^{-1/2}$,
we deduce from the preceding that
\bq\label{eq:amir1}
\kappa_m \ge \frac{\delta}{2} \epsilon^{3/2} \psi(m) \qquad \forall m \ge m_0 \,.
\eq
Now, by the same argument used for bounding $c_2$, we have that
$$
g(x) \ge \eta p^{(2)}_{\lceil M x \rceil^2} \ge \eta (1+M x)^{-1/2} \psi( \lceil M x \rceil^2)^{-1/2} \,.
$$
Fixing $y$ and increasing $m_0$ as needed, for $m \ge m_0$ both
$y \le L(m)$ and $\lceil M L(m) \rceil^2 \le m$. Hence, with
$I(y,z):=\int_y^z (1+M x) \PP(-X_1 > x) dx$ and
$\psi_\star(r):=\sup_{\ell \le r} \psi(\ell)$, it follows that
for $m \ge m_0$,
\begin{align*}
\eta^2 \kappa_m &\le \int_0^{L(m)} (1+M x) \PP(-X_1 > x) \psi( \lceil M x \rceil^2) dx  \\
&\le C(y) + I(y,\infty) \psi_\star(m) \,,
\end{align*}
where $C(y) := \psi_\star ((1+M y)^2) I(0,y)$ is finite for any $y$ finite.
Considering this inequality and (\ref{eq:amir1}), we conclude that
for some $c=c(\delta,\epsilon,M,\eta)$ positive, any $y$ finite and
all $m \ge m_0$ for which $\psi(m)=\psi_\star(m)$,
$$
c \psi_\star (m) \le \eta^2 \kappa_m \le C(y) + I(y,\infty)  \psi_\star(m)
\,.
$$
Finally, with $\EE[ (1+M X_1^-)^2]$ finite, clearly $I(y,\infty) \to 0$ as $y \to \infty$,
hence the preceding inequality implies that $m \mapsto \psi(m)$ is bounded
above. That is, $p_m^{(2)} \ge c_3 m^{-1/4}$ for some $c_3>0$ and all $m \ge 1$,
as claimed in (\ref{quarter-exp}).
}

\end{document}